\newif\ifams\amstrue
\ifams
\documentclass{amsart}
\else
\documentclass[smallextended]{svjour3}
\fi
\ifams\else\usepackage{amsmath,amssymb}\fi

\usepackage{tikz,comment}
\usetikzlibrary{shapes.geometric, arrows.meta, decorations.pathmorphing}

\ifams
\newtheorem{theorem}{Theorem}[section]
\newtheorem{lemma}[theorem]{Lemma}
\newtheorem{proposition}[theorem]{Proposition}

\newtheorem{problem}[theorem]{Problem}
\newtheorem{definition}[theorem]{Definition}
\fi

\newcommand{\tpt}{\mathbf{2}+\mathbf{2}}
\newcommand{\tpo}{\mathbf{3}+\mathbf{1}}

\ifams
\author{Csaba Bir\'o}
\address{Department of Mathematics, University of Louisville, Louisville, KY 40220}

\author{Sida Wan}
\address{Department of Mathematics, University of Louisville, Louisville, KY 40220}
\else
\author{Csaba Bir\'o \and Sida Wan}
\institute{Department of Mathematics, University of Louisville, Louisville, KY 40220}
\fi

\ifams
\title[Dimension bounds on classes of interval orders]{Dimension bounds on classes of interval orders with restricted representation}
\begin{document}
\else
\begin{document}
\title{Dimension bounds on classes of interval orders with restricted representation}
\maketitle
\fi

\begin{abstract}
In general, representations of interval orders may use an arbitrary set of interval lengths. We can define subclasses of interval orders by restricting the allowable lengths of intervals. Motivated by a recent paper of Keller, Trenk, and Young, we study the dimension of posets in some of these subclasses. Among other results, we answer several of their questions, and we simplify the proof of one of their main results.
\end{abstract}

\ifams\maketitle\fi

\section{Introduction}\label{sec:intro}

\subsection{Interval orders}
A \emph{partial order} $P$ is a reflexive, antisymmetric, and transitive relation on a set. A set $X$ together with a partial order $P$ on $X$ is called a \emph{poset}, with \emph{ground set} $X$. In this paper, we assume $X$ is finite. The elements of $X$ are sometimes called \emph{points}, or \emph{vertices}. To avoid clutter, with a slight abuse of notation, we will often use the same letter to denote the poset, its ground set, and the relation.

Let $\mathcal{I}$ be the set of nonempty closed intervals of the real line.
An \emph{interval representation} of $P$ is a function $f:P\to\mathcal{I}$ such that $x<y$ in $P$ if and only if $f(x)$ is entirely left of (and hence disjoint from) $f(y)$. An \emph{interval order} is a poset that has an interval representation.

Restrictions can be placed on the codomain of $f$ to define subclasses of interval orders. Let $S$ be a nonempty set of nonnegative real numbers. An \emph{$S$-representation} of $P$ is an interval representation of $P$ such that the length of $f(x)$ is in $S$ for all $x\in P$.

Following Fishburn and Graham \cite{FG-85}, we will use the notation $C(S)$ to denote the family of posets that have an $S$-representation. As a special case, $C([\alpha,\beta])$ denotes the family of posets for which there is a representation with intervals of lengths between $\alpha$ and $\beta$ (inclusive). We will use the shorthands $C[\alpha,\beta]=C([\alpha,\beta])$, and $C(\alpha)=C([1,\alpha])$.

Obviously, $C[\alpha, \beta] = C[m\alpha, m\beta]$, for all $m\in\mathbb{R}^+$. It is also clear that the class of interval orders is $C(\mathbb{R}^+\cup\{0\})=C(\mathbb{R}^+)$.

A well-studied subclass of interval orders is $C(1)$ which is called the class of \emph{semiorders}. Characterization of interval orders and semiorders via forbidden subposets are classical results of poset theory. Fishburn \cite{F-70} proved that a poset is an interval order, if and only if contains no $\tpt$: a poset consisting of two incomparable chains, each of size $2$. Scott and Suppes \cite{SS-58} proved that a poset is a semiorder if and only if it contains no $\tpt$ and no $\tpo$; this latter is a poset consisting of two incomparable chains, sizes $3$ and $1$.

The \emph{interval count} of the interval order $P$, denoted by $IC(P)$ is the least integer $k$ for which $P$ has an $S$-representation such that $|S|=k$. The interval count problem was suggested by Graham (cf.~\cite{LAP-82}). Semiorders are exactly the posets $P$ for which $IC(P)=1$. However, the problem of deciding whether a poset is of interval count at most $2$ already seems to be a difficult problem. The decision problem is not known to be either NP-complete or polynomial.

According to a survey by Cerioli et al. \cite{COS-12}, ``The interval count problem is an intriguing problem, in which very intuitive statements are proved not to hold.'' For example, for each $k\geq 2$, Fishburn \cite{F} constructed an infinite family of minimal interval orders whose interval count is greater than $k$, so for $k\geq 2$, the class of posets with interval count $k$ does not have a characterization with a finite set of minimal excluded subposets.

\subsection{Dimension}
A set of linear extensions $\{L_1,\ldots,L_t\}$ of the poset $P$ is called a \emph{realizer}, if $P=\cap L_i$. To verify that a certain set of linear extensions is a realizer, it suffices to verify that for every pair $(x,y)$ of incomparable elements there exists a linear extensions $L_i$ in the set in which $x>y$. We will refer to this property as $(x,y)$ being \emph{reversed} in $L_i$. (In fact it suffices to verify that every so-called \emph{critical pair} is reversed, but we will not make use of this property.)

The minimum cardinality of a realizer is the \emph{dimension} of $P$. An extensive treatise of dimension theory of posets is the monograph by Trotter \cite{T}. Since dimension is a classic way of measuring the complexity of a poset, it is natural to ask if the dimension of (restricted) interval orders is bounded.

It has been well-known that the dimension of the class of interval orders is unbounded. In fact, the dimension of the universal interval order $I_n$, whose representation uses every closed interval with integer endpoints between $1$ and $n$, is well-understood. The rate of growth, $\Theta(\lg\lg n)$, was first determined by F\"uredi et al.~\cite{FHRT-92}, but using further results, the exact value is known within an error of $5$.

On the other hand, Rabinovitch \cite{R-78} proved that the dimension of semiorders is at most $3$. A simple proof of this theorem is provided by Bosek et al.~ \cite{BKKM-12}. Similar techniques to their proof is used in this paper.

\subsection{Motivation}
Allowing open intervals of arbitrary length in a representation does not extend the class of interval orders. However, this is not the case with restricted classes.

In a recent paper, Keller, Trenk, and Young \cite{KTY-22} proved the following result.
\begin{theorem}\label{thm:KTY}
If $P$ is in $C(\{0,1\})$, then $\dim(P)\leq 3$.
\end{theorem}

They proposed among others, the following problems (paraphrased).
\begin{problem}\ \label{pr:KTY}
\begin{enumerate}
\item \label{pr:KTY2} Find an upper bound on the dimension for posets in $C(\{r,s\})$ with $r,s>0$.
\item \label{pr:KTYr} More generally, find an upper bound on the dimension for posets in $C(S)$ in terms of $|S|$.
\end{enumerate}
\end{problem}

They offered the easy upper bounds $8$ for (\ref{pr:KTY2}), and $3r+\binom{|S|}{2}$ for (\ref{pr:KTYr}). In Section~\ref{sec:KTYq}, we give better bounds for both parts. In fact, together with a computer search done by K\'ezdy, we will see that the maximum dimension of posets in $C(\{r,s\})$, $r,s>0$ is between $4$ and $5$. Also, using the techniques we develop, we simplify the proof of Theorem~\ref{thm:KTY}, and we prove some other related results.

\section{Twin-free and distinguishing representations}

Let $P$ be an interval order, and fix a representation for $P$. Let $x,y\in P$ be such that the same interval is assigned to both. We call $x$ and $y$ a \emph{twin} (of points). If a representation does not have any twins, we call it \emph{twin-free}. A representation of an interval order is \emph{distinguishing}, if every real number occurs at most once as an endpoint of an interval of the representation, i.e.~no two intervals share an endpoint. A distinguishing representation is of course twin-free.

Let $P$ be a poset, and $x,y\in P$. We say $x$ and $y$ have \emph{duplicated holdings}, if $\{z\in P: z>x\}=\{z\in P: z>y\}$ and $\{z\in P: z<x\}=\{z\in P: z<y\}$; in other words the upsets and the downsets of $x$ and $y$ are the same. If $P$ is an interval order with a representation in which $x$ and $y$ are twins, then they have duplicated holdings. So, if an interval order has no duplicated holdings, then every representation is twin-free.

One important property of two elements with duplicated holdings is that we may discard one of them without reducing the dimension (as long as the dimension is at least $2$). We will use this property later by assuming that a certain poset, for which we are proving an upper bound for its dimension, has no duplicated holdings.

It is easy to see, that every interval order has a distinguishing $\mathbb{R}^+$-representation. Things get less obvious for other $S$-representations. For example, an antichain of size at least $2$ does not have a distinguishing (or even twin-free) $\{0\}$-representation. We prove that---essentially---this is the only problem case.

We caution the reader that the usual handwaving argument of ``just shake the intervals a little until there are no common endpoints'' will not work here. Since we are not allowed the change of the set of lengths, common endpoints may be essential and necessary. The fact that we use closed intervals plays a crucial role here. E.g. $\tpo$ has a $\{1\}$-representation using open and closed intervals, but there is no distinguishing representation.

Let $P$ be a poset, and let $f$ be a representation. In the following proofs and the balance of the paper, we use the notation $\ell_x$ and $r_x$ for the left and right endpoint of $f(x)$. It will also be convenient to assume that $f$ is injective, and its range is a multiset of intervals. Sometimes, this multiset will be refered to as the representation of the poset. With a further slight abuse of notation, this multiset may also be referred as $P$.

\begin{theorem}\label{thm:representations}
Let $S\subseteq\mathbb{R}^+\cup\{0\}$, $S\neq\emptyset$.
\begin{enumerate}
\item\label{part:1} Every poset $P\in C(S)$ that has a twin-free $S$-representation also has a distinguishing $S$-representation.
\item\label{part:2} If $0\not\in S$, then every poset $P\in C(S)$ has a distinguishing $S$-representation.
\end{enumerate} 
\end{theorem}

Before the proof, we define two symmetric operations on the multiset of intervals of a representation. These will preserve the lengths of the intervals, and they will be used to decrease the number of common endpoints.

\begin{definition}
Let $P$ be a poset with a fixed representation. Let $c\in\mathbb{R}$, and $\epsilon>0$. Let $L=\{x\in P: \ell_x<c\}$, and let $R=P-L$. Define $L'=\{[\ell_x+\epsilon,r_x+\epsilon]:x\in P\}$. Let $P'=L'\cup R$, a multiset of intervals. The operation that creates $P'$ from $P$ is what we call ``left pull'' with parameters $c$ and $\epsilon$.

We can similarly define ``right pull''. Let $R=\{x\in P: r_x>c\}$, and let $L=P-R$. Define $R'=\{[\ell_x-\epsilon,r_x-\epsilon]:x\in P\}$. Let $P'=L\cup R'$ to define the operation of right pull.
\end{definition}

\begin{lemma}\label{lem:pull}
Let $P$ be a poset (representation), $c\in\mathbb{R}$, and let $\epsilon=\frac12\min\{|a-b|:\text{$a$ and $b$ are distinct endpoints}\}$. Let $P'$ be the left (right) pull of $P$ with parameters $c$ and $\epsilon$. Then $P$ and $P'$ represent isomorphic posets.
\end{lemma}

\begin{proof}
We will do the proof for left pulls. The argument for right pulls is symmetric.

Notice that if $a$ and $b$ are two endpoints of intervals of $P$, then their relation won't change, unless $a=b$. More precisely, if $a<b$ in $P$ then the corresponding points in $P'$ will maintain this relation. Similarly for $a>b$.

So if $x$ and $y$ are two intervals in $P$ with no common endpoints, then their (poset) relation is maintained in $P'$.

Now suppose that $x$ and $y$ are intervals with some common endpoints. There are a few cases to consider.

If $\ell_x=\ell_y$ then either $x,y\in L$ or $x,y\in R$, so either both are shifted, or neither. Therefore $x\|y$ both in $P$ and in $P'$.

Now suppose $\ell_x\neq \ell_y$; without loss of generality $\ell_x<\ell_y$. Also assume $r_x=r_y$. Then $\ell_x+\epsilon<\ell_y$, so $x\|y$ both in $P$ and in $P'$.

The remaining case is, without loss of generality, $r_x=\ell_y$. Then $r_x+\epsilon<r_y$ (unless $\ell_y=r_y=r_x$, which was covered in the second case), so, again $x\|y$ both in $P$ and in $P'$.
\end{proof}

Now we are ready to prove Theorem~\ref{thm:representations}.

\begin{proof}[Proof of Theorem~\ref{thm:representations}]
Let $S\subseteq\mathbb{R}^+\cup\{0\}$, $S\neq\emptyset$, and let $P \in C(S)$. Consider an $S$-representation of $P$.

We will perform left and right pulls until no common endpoints remain except for twins. Let $x$, $y$ be two intervals with a common endpoint, but $x\neq y$. Let $\epsilon=\frac12\min\{|a-b|:\text{$a$ and $b$ are distinct endpoints}\}$, as in Lemma~\ref{lem:pull}.

\begin{itemize}
\item
If $\ell_x=\ell_y$ and $r_x\neq r_y$, perform a right pull with $c=\min\{r_x,r_y\}$ and $\epsilon$.
\item
If $r_x=r_y$ and $\ell_x\neq \ell_y$, perform a left pull with $c=\max\{\ell_x,\ell_y\}$ and $\epsilon$.
\item If $\ell_x<r_x=\ell_y<r_y$ (or vice versa) either a left or a right pull will work with $c=r_x=\ell_y$.
\end{itemize}

Note that even though the definition of $\epsilon$ looks the same in every step, the actual value will change as the representation changes. Indeed, it is easy to see that $\epsilon$ is getting halved in every step.

If $P$ started with a twin-free representation, then we have arrived to a distinguishing representation, so part~(\ref{part:1}) is proven.

If $P$ had twins, those are still present at the representation. Let $x$ and $y$ be identical intervals of the representation, and let $\epsilon=\frac12\min\{|a-b|:\text{$a$ and $b$ are distinct endpoints}\}$ again. If $0\not\in S$, then the length of $x$ (and hence the length of $y$) is positive. Note that this length is at least $\epsilon$. Move $x$ by $\epsilon$ to the right, that is, replace $x$ with the interval $[\ell_x+\epsilon,r_x+\epsilon]$. The new representation will not have the $x$,$y$ twin and represents the same poset. Repeat this until all twins disappear.
\end{proof}

\section{Choice functions and partitions}

Let $P$ be an interval order with a representation $I$. An injective function $f:\mathbb{R}\to\mathbb{R}$ is a \emph{choice function}, if $f(x)\in I(x)$ for all $x\in P$. Every choice function $f$ defines a linear extension $L(f)$ of $P$ in a natural way: $x<y$ in $L(f)$ if $f(x)<f(y)$.

Choice functions were defined by Kierstead and Trotter \cite{KT-00}. Theorems~\ref{thm:choice} and \ref{thm:partition} appear in their paper, but our proofs in this section are different; in fact we found them before we found their paper. We feel that the proofs here are more transparent and more constructive in the sense, that they lead to easy-to-implement algorithms.

What makes the simple idea of choice functions powerful is the following theorem.

\begin{theorem}\label{thm:choice}
Let $P$ be an interval order with no duplicated holdings, and let $I$ be a distinguishing representation. If $L$ is a linear extension of $P$, then there exists a choice function $f$ with $L(f)=L$.
\end{theorem}

\begin{proof}
Let $|P|=n$, and 
let $\epsilon=\frac{1}{n}\min\{|a-b|:\text{$a$ and $b$ are distinct endpoints}\}$.
Let $L$ be such that $x_1<\cdots<x_n$ in $L$. We will construct $f$ recursively. 

\[
f(x_i)=
\begin{cases}
f(x_{i-1})+\epsilon & \text{if $i\geq 2$ and $f(x_{i-1})\geq \ell_{x_i}$}\\
\ell_{x_i} & \text{otherwise}\\
\end{cases}
\]

Obviously, $f(x_{i-1})<f(x_i)$ for all $i$. So we only need to show that $f(x_i)\in I(x_i)$.

If $f(x_i)=\ell_{x_i}$ this is clear, so we may assume that $f(x_i)\neq \ell_{x_i}$. Call the points $x_i$ with this property ``forced''.

Let $j$ be the greatest positive integer such that $j\leq i$ and $x_j$ is not forced. 
Firstly, $r_{x_i}>\ell_{x_j}$, 
secondly, $r_{x_i}\geq \ell_{x_j}+n\epsilon$.
Hence $f(x_i)=\ell_{x_j}+(i-j)\epsilon\leq \ell_{x_j}+n\epsilon\leq r_{x_i}$, and so
$\ell_{x_i}\leq f(x_{i-1})\leq f(x_i)\leq r_{x_i}$.
\end{proof}

Theorem~\ref{thm:representations} and Theorem~\ref{thm:choice} show that constructing linear extensions for posets in classes $C(S)$ can usually be done by constructing choice functions. This makes choice functions a useful tool for upper bound proofs on the dimension for these classes. We will demonstrate this repeatedly in the balance of the paper.

The following theorem shows how special interval orders are. Nothing remotely close is true for general posets.

\begin{theorem}\label{thm:partition}
Let $P$ be an interval order, and partition the ground set $X$ into $s$ parts: $X=X_1\cup\cdots\cup X_s$. Let $L_i$ be a linear extensions of $P|_{X_i}$. Then there exists a linear extension $L$ of $P$ such that $L|_{X_i}=L_i$.
\end{theorem}

\begin{proof}
The $s=1$ case is trivial. For $s\geq 2$, we will use induction on $s$. Let $s=2$.

Let $X_1,X_2,L_1,L_2$ be defined as in the lemma. Define the relation $E=L_1\cup L_2\cup  P$, and the directed graph $G=(X,E)$; that is, $G$ is the directed graph fromed by the $<$ relation in $P$, $L_1$, and $L_2$. It is sufficient to show that $G$ has no directed closed walk; indeed, if that is the case, the transitive closure $T$ of $G$ is an extension of the poset $P$, and any linear extension $L$ of $T$ will satisfy the requirements of the conclusion of the lemma.

Suppose for a contradiction that $G$ contains a directed closed walk.
Since neither $G[X_1]$ nor $G[X_2]$ contains a directed closed walk, every directed closed walk in $G$ must have both an $X_1X_2$ and an $X_2X_1$ edge. We will call these edges \emph{cross-edges}. Let $C$ be a directed closed walk in $G$ with the minimum number of cross-edges.

As we noted, $C$ contains at least one $X_1X_2$ edge; let $(a,b)$ be such an edge. Let $(c,d)$ be the first $X_2X_1$ edge that follows $(a,b)$ in $C$. Observe that $c<d$, $a<b$ in $P$, and $b\leq c$ in $L_2$. If $d=a$, then $c<d=a<b$ in $P$, which would contradict $b\leq c$ in $L_2$. If $d>a$ in $L_1$, then we could eliminate the path $ab\ldots cd$ in $C$, replacing it with the single-edge path $ad$, and thereby decreasing the number of cross-edges in $C$, contradicting the minimality of $C$. (See Figure~\ref{fig:cycles}.)

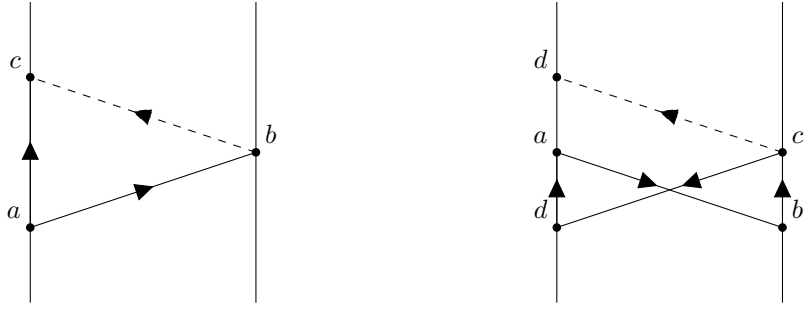
\begin{figure}
\begin{tikzpicture}
    \tikzstyle{oval} = [ellipse, minimum width=2cm, minimum height=4cm, draw]
    \tikzstyle{point} = [circle, minimum size=4pt, inner sep=0pt, fill, draw]

    \node[oval] (leftOval) at (0,0) [label=$X_1$] {};
    \node[oval] (rightOval) at (3,0) [label=$X_2$] {};

    \coordinate (leftTop) at ([yshift=1.5cm]leftOval.center);
    \coordinate (rightTop) at ([yshift=1.5cm]rightOval.center);
    \coordinate (leftBottom) at ([yshift=-1.5cm]leftOval.center);
    \coordinate (rightBottom) at ([yshift=-1.5cm]rightOval.center);

    \draw[-{Latex[length=3mm]}] (leftTop) to[bend left] (rightTop);
    \draw[-{Latex[length=3mm]}] (rightBottom) to[bend left] (leftBottom);

    \draw[-{Latex[length=3mm]}, decorate, decoration={snake, amplitude=.4mm, segment length=2mm, post length=3mm}] (rightTop) to[bend left=20] (rightBottom);
    
    \draw[-{Latex[length=3mm]}, decorate, decoration={snake, amplitude=.4mm, segment length=2mm, post length=3mm}] plot [smooth, tension=1] coordinates {(leftBottom) (3,0) (leftTop)};
    
    \draw[-{Latex[length=3mm]},dashed] (leftBottom) to[bend left=20] (leftTop);
    
    \node[point] at (leftTop) [label=left:$a$] {};
    \node[point] at (rightTop) [label=right:$b$] {};
    \node[point] at (rightBottom) [label=right:$c$] {};
    \node[point] at (leftBottom) [label=left:$d$] {};
\end{tikzpicture}
\caption{Minimal oriented cycles}\label{fig:cycles}
\end{figure}

So we conclude that $d<a$ in $L_1$, and recall that $b\leq c$ in $L_2$. If $b\leq c$ in $P$, then $a<b\leq c<d$ would contradict $d<a$ in $L_1$. (In particular, $b\neq c$.) Obviously, $b\not>c$ in $P$, so $b\|c$ in $P$. Similarly, $d\| a$ in $P$. Hence the set $\{a,b,c,d\}$ induces a $\mathbf{2}+\mathbf{2}$ in $P$, a contradiction.

If $s>2$, then we can apply the hypothesis for the part $X'=X_1\cup\cdots\cup X_{s-1}$, and use the $s=2$ case for $X'$ and $X_s$. This finishes the proof.
\end{proof}

We conclude the section with the following all-important corollary, which is a generalization of a theorem by Kierstead and Trotter~\cite{KT-00}.

\begin{theorem}\label{thm:maxplus2}
Let $P$ be an interval order, and partition the ground set $X$ into $s$ parts: $X=X_1\cup\cdots\cup X_s$. Let $P_i=P|_{X_i}$. Then
\[
\dim(P)\leq \max\{\dim(P_1),\ldots,\dim(P_s)\}+2\lceil \lg s\rceil
\]
\end{theorem}

\begin{proof}
Let $t=\max\{\dim(P_1),\dim(P_2)\}$. We proceed
by induction on $s$. Trivial for $s=1$; let $s=2$. Consider a distinguishing representation of $P$.

By Theorem ~\ref{thm:partition}, there exists a family $\mathcal{R}$ of $t$ linear extensions of $P$, such that the restriction of the linear extensions in $\mathcal{R}$ to each $X_i$ form a realizer of $P_i$. Then define two choice functions $f_1$ and $f_2$, where $f_1(x)=\ell_x$, $f_2(x)=r_x$ for every $x \in X_1$; $f_1(y)=r_y$, $f_2(y)=\ell_y$ for every $y \in X_2$. Let $L_1=L(f_1)$, $L_2=L(f_2)$. Clearly, $\mathcal{R} \cup \{L_1, L_2\}$ is a realizer of $P$.

If $s>2$, then we can apply the hypothesis for the parts $X'=X_1\cup\cdots\cup X_{\lfloor s/2\rfloor}$ and $X''=X_{\lfloor s/2\rfloor+1}\cup\cdots\cup X_s$, and use the $s=2$ case for $X'$ and $X''$. A routine calculation finishes the proof.
\end{proof}

\section{The Keller--Trenk--Young Theorem}

This section contains a short proof of Theorem~\ref{thm:KTY}.

Let $P$ be a poset. One can partition the ground set $X$ of $P$ into antichains by subsequent removal of minimal elements. That is, let $A_1=\min(P)$, $A_2=\min(P-A_1)$, \dots, $A_i=\min(P-A_1-\cdots-A_{i-1})$. Then $X=A_1\cup\cdots\cup A_h$ is a partition, $h$ is the height of $P$, and each $A_i$ is an antichain. We will call this the \emph{ranking partition} of $P$, as it defines a rank function on graded posets. (But we may apply this definition for non-graded posets, as well.)

The following simple lemma has been in used in many arguments.

\begin{lemma}\label{lem:ranking}
If $P$ has no $\tpo$ subposet, and $A_1,\ldots,A_h$ is a ranking partition, then $A_i<A_{i+2}$ for all $i$.
\end{lemma}

\begin{proof}
Let $x\in A_i$, $y\in A_{i+2}$. Then there exists $z\in A_{i+1}$ with $z<y$, and $w\in A_i$ with $w<z$. If $x>y$, then $x>y>z>w$, which is not possible, because $x,w\in A_i$.
If $x\|y$, then the set $\{x,y,z,w\}$ induces a $\tpo$. So $x<y$.
\end{proof}

We repeat Theorem~\ref{thm:KTY} from Section~\ref{sec:intro} for the convenience of the reader.

{
\renewcommand{\thetheorem}{\ref{thm:KTY}}
\begin{theorem}
If $P$ is in $C(\{0,1\})$, then $\dim(P)\leq 3$.
\end{theorem}
\addtocounter{theorem}{-1}
}

\begin{proof}
We may assume that $P$ has no duplicated holdings, so it has a twin-free $\{0,1\}$-representation. By Theorem~\ref{thm:representations}, we may consider a distinguishing representation. Partition $P$ into $U\cup Z$, where $U$ contains the points that are represented by unit intervals, and $Z$ contains the ones whose interval length is $0$. Let $A_1,\ldots,A_h$ be a ranking partition of $P|_U$. Let $E=\cup\{A_i:\text{$i$ is even}\}$, and $O=\{A_i:\text{$i$ is odd}\}$.

Define two choice functions:
\begin{align*}
f_1(x)&=
\begin{cases}
\ell_x&\text{ if $x\in O\cup Z$}\\
r_x&\text{ if $x\in E$}
\end{cases}\\
f_2(x)&=
\begin{cases}
\ell_x&\text{ if $x\in E\cup Z$}\\
r_x&\text{ if $x\in O$}
\end{cases}
\end{align*}
Let $L_1=L(f_1)$, and $L_2=L(f_2)$. Define a third linear extension:
\[
L_3:(L_1|_{A_1})^d<\cdots<(L_1|_{A_h})^d
\]

The linear extensions $L_1$ and $L_2$ will reverse every incomparable pair for which (1) one is in $U$, the other is in $Z$; (2) one is in $E$, the other is in $O$. The only incomparable pairs left are in the same $A_i$: those are reversed in $L_3$.
\end{proof}

\section{Dimension bounds for posets of bounded interval count}
\label{sec:KTYq}

Since $C(\{0,s\})=C(\{0,1\})$ for all $s>0$, Theorem~\ref{thm:KTY} provides a perfect answer to dimension bounds on this class. The next natural question, also question (\ref{pr:KTY2}) of Problem~\ref{pr:KTY}, is to find an upper bound for the dimension of posets in the class $C(\{r,s\})$ with $r,s>0$. We give an answer below, using the terminology of interval counts.

\begin{proposition}\label{prop:KTY2}
If $P$ be an interval order with $IC(P)\leq 2$. Then $\dim(P)\leq 5$.
\end{proposition}

\begin{proof}
Consider an $\{r,s\}$-representation, and partition $P$ into $X_1\cup X_2$ with $X_1$ containing points corresponding to intervals of length $r$, and $X_2$ containing points corresponding to intervals of length $s$. Since $\dim(P|_{X_1})\leq 3$ and $\dim(P|_{X_2})\leq 3$, using Theorem~\ref{thm:maxplus2} for $s=2$, we conclude $\dim(P)\leq 5$.
\end{proof}

Andr\'e K\'ezdy (personal communication) ran a computer search using a state-of-the-art dimension algorithm to find a large dimensional poset of interval count $2$.
K\'ezdy found that the interval order represented by the $51$ intervals in the set
\[
\{[a,b]:\quad a,b\in\mathbb{Z};\quad 1\leq a,b\leq 30;\quad|b-a|\in\{1,8\}\}
\]
is $4$-dimensional. Furthermore, he found that it has a $45$-element subposet that is irreducible, i.e.\ the removal of any point reduces the dimension to $3$. (This poset can be constructed from the set of intervals above by removing the intervals $[8, 9]$, $[9, 10]$, $[10, 11]$, $[20, 21]$, $[21, 22]$, $[22, 23]$.)

This answers another question by Keller, Trenk, and Young, namely, the minimum number of interval lengths required to force a $4$-dimensional interval order.

Using the general version of Theorem~\ref{thm:maxplus2}, we conclude the following result, providing a logarithmic upper bound to question (\ref{pr:KTYr}) of Problem~\ref{pr:KTY}.

\begin{proposition}
Let $P$ be an interval order with $IC(P)\leq r$. Then $\dim(P)\leq 2\lceil\lg r\rceil+3$.
\end{proposition}

Keller, Trenk, and Young conjecture an upper bound of $O(\lg\lg r)$ for the dimension of intervals orders of count at most $r$. They base their conjecture on the dimension of the universal interval orders. However, we note that the universal interval order is a much stronger restriction than a bounded interval count. Even the rate of growth of the best upper bound remains wide open.

\section{Posets in $C(t)$}

Recall that $C(t)$ is the class of interval orders representable by intervals of length between $1$ and $t$. In this sections, we use our techniques to find upper bounds for this class.

\begin{theorem}\label{thm:C2}
If $P$ is in $C(2)$, then $\dim(P)\leq 4$.
\end{theorem}

We start with a generalization of Lemma~\ref{lem:ranking}. The proof is essentially the same.

\begin{lemma}\label{lem:genranking}
If $P$ has no $\mathbf{k}+\mathbf{1}$ subposet, and $A_1,\ldots,A_h$ is a ranking partition, then $A_i<A_{i+k-1}$ for all $i$.
\end{lemma}

\begin{proof}[Proof of Theorem~\ref{thm:C2}]
By Theorem~\ref{thm:representations}, $P$ has a distinguishing $[1,2]$-representation. Let $A_1,\ldots,A_h$ be a ranking partition of $P$. Notice that $P$ has no $\mathbf{4}+\mathbf{1}$ subposet, so by Lemma~\ref{lem:genranking}, $A_i<A_{i+3}$ for all $i$. We define three choice functions $f_0,f_1,f_2$ as follows.
\[
f_i(x)=
\begin{cases}
r_x&\text{ if $x\in A_j$ with $j\equiv i\mod 3$}\\
\ell_x&\text{otherwise}
\end{cases}
\]
Let $L_0=L(f_0)$, $L_1=L(f_1)$, $L_2=L(f_2)$. Let $(x,y)$ be an incomparable pair such that $x\in A_i$, $y\in A_j$, $i\neq j$. If $\ell_x<\ell_y$, then they are reversed in $L_k$ for which $k\equiv i\mod 3$. If $\ell_x>\ell_y$, then they are reversed in $L_k$ for which $k\equiv j\mod 3$.

We still need to reverse incomparable pairs that appear within the same $A_i$. This can be done the usual way:
\[
L_3:(L_1|_{A_1})^d<\cdots<(L_1|_{A_h})^d
\]

Now $\{L_0,L_1,L_2,L_3\}$ is a realizer of $P$.
\end{proof}

A simple generalization of this proof would lead to an upper bound $\lceil t\rceil+2$ for the dimension of the class $C(t)$. However, we can do much better with a divide-and-conquer technique.

\begin{theorem}\label{thm:Ct}
For $t\geq 2$, if $P$ is in $C(t)$, then $\dim(P)\leq 2\lceil\lg\lg t\rceil+4$.
\end{theorem}

\begin{proof}
Let $n=2^{2^{\lceil\lg\lg t\rceil}}$. Since $n\geq t$, we have $P\in C(n)$. We will show, by induction on $k$, that for $n=2^{2^k}$ for some $k\in\mathbb{N}$, and for $P\in C(n)$, we have $\dim(P)\leq 2\lg\lg n+4=2k+4$. This will finish the proof.

For $k=0$, the statement is Theorem~\ref{thm:C2}. Now assume $k\geq 1$. Then $n$ is a perfect square; let $m=\sqrt{n}=2^{2^{k-1}}$. Consider a $[1,n]$-representation of $P$, and partition its ground set into points represented by ``short'' intervals, which are of length at most $m$, and ``long intervals'', which are longer than $m$. Let the subsets be $S$ and $L$, respectively.

Notice that $P|_S\in C(m)$, and $P|_L\in C(m,n)=C(n/m)=C(m)$. Using Theorem~\ref{thm:maxplus2}, we have
\[
\dim(P)\leq\max\{\dim(P|_S),\dim(P|_L)\}+2=2\lg\lg m+4+2=2(k-1)+6=2k+4.
\]
\end{proof}

We note that the universal interval order shows that Theorem~\ref{thm:Ct} is best possible in terms of rate of growth (up to a multiplicative factor).

Again, lower bounds in Theorems~\ref{thm:C2} and \ref{thm:Ct} seem to be difficult to find. The existence of a $4$-dimensional poset in $C(2)$ seems likely, but we have not succeeded in finding one. We could only show a very modest result: just one longer interval is not sufficient to raise the dimension to $4$. This is expressed in the following theorem, which is also relevant for finding lower bounds for Proposition~\ref{prop:KTY2}.

\begin{theorem}
Let $P$ be an interval order with a $\{1,s\}$-representation, such that $0\leq s\leq 2$, and there is only one interval of length $s$. Then $\dim(P)\leq 3$.
\end{theorem}

\begin{proof}
Let $x$ be the point represented by the interval of length $s$. With a shift of the intervals in the representation, we may assume that $r_x=-s/2$, and $\ell_x=s/2$, i.e.\ the midpoint of the interval of $x$ is $0$.

Partition the poset based on the position of the rest of the intervals in the representation. Let $U$ be the point corresponding to intervals with positive endpoints, $D$ of the same with negative endpoints, and let $M$ be the set of points with intervals containing the number $0$.

Let $U_1,\ldots,U_m$ be a ranking partition of $P|_U$, and let $D_1,\ldots,D_k$ be a ranking partition of $(P|_D)^d$. Now the ground set of $P$ is partitioned into the sets
\[
D_k,D_{k-1},\ldots,D_1,M,U_1,\ldots,U_m.
\]
Reindex this partition so that the same parts, in the same order, are denoted by $S_1,\ldots,S_{m+k+1}$. We claim that $S_i<S_{i+2}$ for all $i$.

Since $P|_D$ and $P|_U$ are semiorders, and hence have no $\tpo$, by Lemma~\ref{lem:ranking}, it is sufficient to prove that $x<y$ whenever (1) $x\in D_2$, $y\in M$; or (2) $x\in D_1$, $y\in U_1$; or (3) $x\in M$, $y\in U_2$. Of these, two is clear, because $x$ has negative right endpoint, and $y$ has positive left endpoint. The proof of (1) and (3) is symmetric, so we assume $x\in M$ and $y\in U_2$.

Since $y\in U_2$, there exists $z\in U_1$ with $z<y$. Hence
\[
r_x=s/2\leq 1<\ell_z+1=r_z<\ell_y,
\]
which implies $x<y$.

Now we can finish the proof in the familiar way. We define two choice functions: $f_1(x)=\ell_x$, $f_2(x)=r_x$ if $x\in S_{2i}$, and $f_1(x)=r_x$, $f_2(x)=\ell_x$ if $x\in S_{2i+1}$ for some $i$. Construct the two corresponding linear extensions $L_1=L(f_1)$, $L_2=L(f_2)$. Finally, let $L_3$ be such that $(L_1|_{S_1})^d<\cdots<(L_1|_{S_{m+k+1}})^d$ in $L_3$. Then $\{L_1,L_2,L_3\}$ is a realizer.
\end{proof}

\bibliographystyle{plain}
\bibliography{dimbounds}

\ifams\else
\section*{Statements and Declarations}

The authors declare that no funds, grants, or other support were received during the preparation of this manuscript.

The authors have no relevant financial or non-financial interests to disclose.
\fi

\end{document}